\documentclass[11pt]{article}

\usepackage{amsthm,amsmath,stmaryrd,bbm,geometry,color}
\usepackage{amssymb}
\usepackage[english]{babel}
\usepackage[utf8]{inputenc}
\usepackage{graphicx}
\usepackage{verbatim}
\usepackage{enumitem}
\usepackage{mathtools}
\usepackage[titletoc]{appendix}
\usepackage{bm}
\usepackage{authblk}

\usepackage[dvipsnames]{xcolor}
\usepackage[hyperindex=true,frenchlinks=true,colorlinks=true,
citecolor=Mahogany,linkcolor=Red,urlcolor=Tan,linktocpage]{hyperref}

\usepackage{tikz}
\usetikzlibrary{shapes}
\usetikzlibrary{fit}
\usetikzlibrary{decorations.pathmorphing}

\title{Functional central limit theorem and Marcinkiewicz strong law of large numbers 
for Hilbert-valued $U$-statistics of absolutely regular data}
\author{Davide Giraudo}

\date{\today}
\affil[$\dagger$]{Institut de Recherche Mathématique Avancée
UMR 7501, Université de Strasbourg and CNRS
7 rue René Descartes
67000 Strasbourg, France}

\numberwithin{equation}{section}
\renewcommand{\leq}{\leqslant}
\renewcommand{\geq}{\geqslant}

\newtheorem{Theorem}{Theorem}[section]
\newtheorem{Proposition}[Theorem]{Proposition}
\newtheorem{Lemma}[Theorem]{Lemma}
\newtheorem{Definition}[Theorem]{Definition}

\theoremstyle{remark}

\tikzstyle{Vertex}=[circle,draw=LimeGreen!80,fill=LimeGreen!8,
inner sep=1pt,minimum size=2mm,line width=1pt,font=\scriptsize]
\tikzstyle{Node}=[Vertex,draw=RoyalBlue!80,fill=RoyalBlue!8,inner sep=1.5pt]
\tikzstyle{Leaf}=[rectangle,draw=Black!70,fill=Black!16,
inner sep=0pt,minimum size=1mm,line width=1.25pt]
\tikzstyle{Edge}=[Maroon!80,cap=round,line width=1pt]
\tikzstyle{Mark1}=[draw=BrickRed!80,fill=BrickRed!8]
\tikzstyle{Mark2}=[draw=BurntOrange!80,fill=BurntOrange!8]
\tikzstyle{EdgeRew}=[->,RedOrange!80,cap=round,thick]

\newcommand{\Aca}{\mathcal{A}}
\newcommand{\Bca}{\mathcal{B}}

\newcommand{\Fca}{\mathcal{F}}

\newcommand{\Sca}{\mathcal{S}}
\newcommand{\Uca}{\mathcal{U}}

\newcommand \ens[1]{\left\{ #1\right\}}

\newcommand \R{\mathbb R}

\newcommand \N{\mathbb N}
\newcommand \PP{\mathbb P}

\newcommand{\E}[1]{\mathbb E\left[#1\right]}

\newcommand \Z{\mathbb Z}

\newcommand \abs[1]{\left|#1\right|}
\newcommand \eps{\varepsilon}

\newcommand{\pr}[1]{\left(#1\right)}
\newcommand{\norm}[1]{\left\lVert #1 \right\rVert}

\newcommand{\ind}[1]{\mathbf{1}_{#1}}

\newcommand{\ent}[1]{\left\lfloor #1\right\rfloor}

\newcommand{\Hi}{\mathbb{H}}
\newcommand{\scal}[2]{\left\langle #1,#2\right\rangle}
\allowdisplaybreaks
\begin{document}
\maketitle
\begin{abstract}
In this paper, we investigate the functional central limit theorem and the Marcinkiewicz strong law of large numbers for $U$-statistics having absolutely regular data and taking value in a separable Hilbert space. The novelty of our approach consists in using coupling in order to formulate a deviation inequality for original $U$-statistic, where the upper bound involves the mixing coefficient and the tail of several $U$-statistics of i.i.d.\ data. 
The presented results improve the known results in several directions: the case of metric space valued data is considered as well as Hilbert space valued, and the mixing rates are less restrictive in a wide range of parameters.
\end{abstract}

\section{Main results}

In all this paper, we will consider $U$-statistics of order two taking values in a separable Hilbert space $\Hi$ defined as follows: given a sequence of random variables $\pr{X_i}_{i\geq 1}$   taking values in a separable metric space $\pr{S,d}$ and $h\colon S\times S\to \Hi$ is measurable,  
\begin{equation}
U_n\pr{h}=\sum_{1\leq i<j\leq n}h\pr{X_i,X_j}.
\end{equation}

Halmos \cite{MR0015746} and Hoeffding \cite{MR0026294} showed that when $\pr{X_i}_{i\geq 1}$ is i.i.d., $\Hi=\R$ and $\E{\abs{h\pr{X_1,X_2}}}<\infty$, $U_n\pr{h}/\binom n2$ is a consistent estimator of $\E{ h\pr{X_1,X_2}}$. This results has been then extended to Hilbert valued case by Borovskikh in \cite{MR0968922}. Asymptotic normality 
has also been established under the assumption $\E{h\pr{X_1,X_2}^2}<\infty$ in  \cite{MR0026294}. Convergence rates in the central limit theorem for Hilbert space valued $U$-statistics were also considered in \cite{MR1337065}.

This paper is devoted to the obtention of a functional central central limit theorem and Marcinkiewicz strong law of large numbers for $U$-statistics whose 
data is a strictly stationary sequence and taking values in a separable Hilbert space. 
The motivation behind the consideration of vector-valued $U$-statistics 
is to consider spatial sign for robust tests (see \cite{MR3650400,wegner2023robust,MR4499388}), or Wilcoxon-Mann-Whitney-type test
\cite{MR3335109}. 
Moreover, the assumption that the random variables $X_i$ take values in a metric space
instead of the real line is useful
 in order to consider 
function space analogue of Gini’s mean difference (see Section~3.2 in \cite{MR3514512}), or correlation dimension for metric space valued data (cf. \cite{vandelft2023statistical}).
Moreover, high dimensional or functional data can be treated via the use of Hilbert space valued $U$-statistics, see for instance \cite{math11010161,MR4573388}. 
Kendall's tau for functional data can be also treated via the use of Hilbert value $U$-statistics, see \cite{MR4233414}.
We refer the reader to the book \cite{MR0889097} for a complete description of $U$-statistics taking values in Hilbert spaces.

The classical strong law of large numbers for $U$-statistics of strictly stationary data without dependence condition has been considered in \cite{MR1363941,MR1687339,dehling2023remarks}.
A strong law of large numbers has been established in \cite{MR2915087} 
for $U$-statistics of arbitrary order whose data comes from a Markov chain.
Arcones \cite{MR1624866} showed a law of large numbers for $U$-statistics of order $m$, where the normalisation is $n^m$. A similar result has been obtained for $2m$-wise independent sequences, that is, sequences $\pr{X_i}_{i\geq 1}$ such that 
for each $i_1<\dots<i_{2m}$, the random variables $X_{i_1},\dots,X_{i_{2m}}$ are independent. Related results on $V$-functionals of $\alpha$-mixing data were estalished in \cite{MR3259868}.
 In \cite{MR4243516}, the law of large number for $U$-statistics of order two whose data comes from a function of an i.i.d.\ sequence has been investigated.

In order to study the asymptotic behavior of $U_n\pr{h}$, we decompose the kernel as follows: take an independent copy $X'_1$ of $X_1$ and let 
\begin{equation}\label{eq:def_h10}
h_{1,0}\pr{x}=\E{h\pr{x,X'_1}}-\E{h\pr{X_1,X'_1}}
\end{equation}
\begin{equation}\label{eq:def_h01}
h_{0,1}\pr{y}=\E{h\pr{X_1,y}}-\E{h\pr{X_1,X'_1}}
\end{equation}
and
\begin{equation}\label{eq:def_h2}
h_2\pr{x,y}=h\pr{x,y}-h_{1,0}\pr{x}-h_{0,1}\pr{y}-\E{h\pr{X_1,X'_1}}.
\end{equation}
In this way,
\begin{equation}\label{eq:Hoeffding}
U_n\pr{h}
= \sum_{i=1}^{n-1}\pr{n-i}\pr{h_{1,0}\pr{X_i}-\E{h_{1,0}\pr{X_i}}}
+ \sum_{j=2}^{n}\pr{j-1}\pr{h_{0,1}\pr{X_j}-\E{h_{0,1}\pr{X_j}}}
+   U_n\pr{h_2}
\end{equation}
and when $\pr{X_i}_{i\geq 1}$ is i.i.d.\ , one has 
\begin{equation}
\E{h_2\pr{X_i,X_j}\mid\sigma\pr{X_k,k\leq j-1}}=0
=\E{h_2\pr{X_i,X_j}\mid\sigma\pr{X_k,k\geq i+1}}
\end{equation}
hence martingale and reversed martingale properties can be used in order to control moments of the sum over $j$ and $i$ respectively. 

It is also worth pointing out that when $h$ is symmetric, that is, $h\pr{x,y}=h\pr{y,x}$ for each $x,y\in S$, the functions $h_{1,0}$ and $h_{0,1}$ coincide hence 
\eqref{eq:Hoeffding} admits the simpler form 
\begin{equation}
\label{eq:Hoeffding_sym}
U_n\pr{h}
= n\sum_{k=1}^{n} \pr{h_{1,0}\pr{X_k}-\E{h_{1,0}\pr{X_k}}}+   U_n\pr{h_2}.
\end{equation}

The dependence in our results will be quantified by the so-called mixing coefficients.
Let $\pr{\Omega,\Fca,\PP}$ be a probability space. 
The $\alpha$-mixing and $\beta$-mixing coefficients between two sub-$\sigma$-algebras 
$\Aca$ and $\Bca$ of $\Fca$ are defined respectively by 

\begin{equation}
 \alpha\pr{\Aca,\Bca}=\sup\ens{
 \abs{\PP\pr{A\cap B}-\PP\pr{A}\PP\pr{B}},A\in\Aca, B\in \Bca
 };
\end{equation}
\begin{equation}
 \beta\pr{\Aca,\Bca}=\frac 12\sup\ens{
 \sum_{i=1}^I\sum_{j=1}^J\abs{\PP\pr{A_i\cap B_j}-
 \PP\pr{A_i}\PP\pr{B_j}}},
\end{equation}
where the supremum runs over all the partitions 
$\pr{A_i}_{i=1}^I$ and $\pr{B_j}_{j=1}^J$ of $\Omega$ of 
elements of $\Aca$ and $\Bca$ respectively.
Given a strictly stationary sequence $\pr{X_i}_{i\geq 1}$, we associate its 
sequences of $\alpha$ and $\beta$-mixing coefficients 
by letting 
\begin{equation}
\alpha\pr{k}:=\sup_{\ell\geq 1}
\alpha\pr{\Fca_1^\ell,\Fca_{\ell+k}^{\infty}},
\end{equation}
\begin{equation}
\beta\pr{k}:=\sup_{\ell\geq 1}
\beta\pr{\Fca_1^\ell,\Fca_{\ell+k}^{\infty}},
\end{equation}
where $\Fca_u^v$, $1\leq u\leq v\leq +\infty$ is 
the $\sigma$-algebra generated by the random variables 
$X_i$, $u\leq i\leq v$ ($u\leq i$ for $v=\infty$). A sequence $\pr{X_i}_{i\geq 1}$ is said to be absolutely regular if $\lim_{k\to\infty}\beta\pr{k}=0$.

The paper is organized as follows: in Subsection~\ref{subsec:TLCF}, we will formulate a 
result for a partial sum process built on a $U$-statistic off absolutely regular data and 
in Subsection~\ref{subsec:SLLN}, results on the strong law of large numbers, distinguishing the degenerated and non-degenerated cases. The proofs are given 
in Section~\ref{sec:proof} and are a consequence of an inequality given in Subsection~\ref{subsec:dev_ineg}.
\subsection{Functional central limit theorem}\label{subsec:TLCF}

Define the process 
\begin{equation}
\Uca_{n,h}\pr{t}=
\sum_{1\leq i<j\leq\ent{nt}}h\pr{X_i,X_j}+\pr{nt-\ent{nt}}
\sum_{i=1}^{\ent{nt}}h\pr{X_i,X_{\ent{nt}+1}}.
\end{equation}
Notice that $\Uca_{n,h}\pr{k/n}=U_{k}\pr{h}$ hence the process $\Uca_{n,h}$ contains the information of all the values of $U_k\pr{h}$, $2\leq k\leq n$. Moreover,  for each $\omega\in\Omega$, the map $t\mapsto \Uca_{n,h}\pr{t}$ belongs to $C_{\Hi}[0,1]$, the space of $\Hi$-valued continuous functions defined on the unit interval endowed with the norm $\norm{x}_{\infty}=\sup_{t\in [0,1]}\norm{x\pr{t}}_{\Hi}$, because we interpolate linearly between the points $\pr{k/n,U_{k}\pr{h}}$. 

The limiting process will be described as follows.

\begin{Definition}
We say that a non-negative self-adjoint operator $\Gamma\colon\Hi\to\Hi$ is an $\Sca\pr{\Hi}$ operator if for some Hilbert basis $\pr{e_i}_{i\geq 1}$, $\sum_{i=1}^\infty\scal{\Gamma e_i}{e_i}_{\Hi}<\infty$.
\end{Definition}
\begin{Definition}
Let $\pi_t\colon C_{\Hi}[0,1]$ be the projection map, that is, $\pi_t\pr{x}=x\pr{t}$.  For $\Gamma\in\Sca\pr{\Hi}$, denote by 
$W_\Gamma$ the process   such that 
\begin{enumerate}[label=\arabic*]
\item $W_\Gamma\pr{0}=0$,
\item for all $0\leq t_0<t_1<\dots<t_N\leq 1$, the increments
$\pr{W\pr{t_i}-W\pr{t_{i-1}}}_{i=0}^N$ are independent and $W\pr{t_i}-W\pr{t_{i-1}}$ has a Gaussian distribution on $\Hi$ with mean zero and covariance operator 
$\pr{t_i-t_{t-1}}\Gamma$.
\end{enumerate}
\end{Definition}

We are now in position to state our result on functional central limit theorem for $U$-statistics of absolutely regular data taking values in a separable Hilbert space.

\begin{Theorem}\label{thm:TLFC}
Let $\pr{X_i}_{i\in\Z}$ be a strictly stationary sequence of random variables taking values in a separable metric space $\pr{S,d}$ and let $h\colon S\times S\to \Hi$ be a symmetric measurable function, where $S\times S$ is endowed with the product $\sigma$-algebra and the Hilbert space $\pr{\Hi,\scal{\cdot}{\cdot}_{\Hi}}$ with the $\sigma$-algebra induced by the norm.  Suppose that the following conditions are satisfied:
\begin{enumerate}[label=(C.\arabic*)]
\item the following series is finite:
\begin{equation}\label{eq:cond_melange_TLCF}
\sum_{k= 1}^\infty \int_0^{\alpha\pr{\sigma\pr{X_i,i\leq 0},\sigma\pr{X_k}}}Q^2_{\norm{h_1\pr{X_0}}_{\Hi}}\pr{u}du<\infty,
\end{equation}
where $h_1\pr{x}=\E{h\pr{x,X_1}}-\E{h\pr{X_1,X'_1}}$ and $X'_1$ is an independent 
copy of $X_1$,
\item \label{assum:mixing_TLFC}$\lim_{n\to\infty}n^2\beta\pr{n}=0$,
\item $\sup_{j\geq 2}\E{\norm{h\pr{X_1,X_j}}_{\Hi}}<\infty$.
\end{enumerate}
Then the following convergence in distribution in $C_{\Hi}[0,1]$ takes place:
\begin{equation}\label{eq:TLCF}
\frac 1{n^{3/2}}\pr{\Uca_{n,h}\pr{t}-\E{\Uca_{n,h}\pr{t}}}\to W_{\Gamma},
\end{equation}
where the operator $\Gamma$ is given by 
\begin{equation}\label{eq:def_operateur_Gamma_TLCF}
\left\langle \Gamma u,v\right\rangle_{\Hi}=\sum_{k\in\Z}\E{\left\langle h_1\pr{X_0},u\right\rangle_{\Hi} \left\langle h_1\pr{X_k},v\right\rangle_{\Hi}}, \quad u,v\in\Hi.
\end{equation}
\end{Theorem}

Notice that our result also applies when $\E{h\pr{X_1,X'_1}\mid X_1}=\E{h\pr{X_1,X'_1}\mid X'_1}=0$ almost surely, in which case the limiting process $W_\Gamma$ is zero because so is $\Gamma$.  
In this case, the appropriated normalization is $n^{-1}$ instead of $n^{-3/2}$ and 
completely different techniques have to be used, like in \cite{MR2922461}.

Let us compare this result with existing ones in the literature. 
Yoshihara \cite{MR0418179} obtained a similar result for real-valued $U$-statistics, but at the cost of a more restrictive assumption on moments and the decay of mixing coefficients. 
Moreover, we address the Hilbert-valued case.

Dehling and Wendler  \cite{MR2557623} obtained a central limit theorem assuming the existence of a positive $\delta$ such that $\sup_{j\geq 2}\E{\norm{h\pr{X_1,X_j}}_{\Hi}^{2+\delta}}<\infty$ and that for some positive $C,\eta$, the $\beta$-mixing coefficient satisfies $\beta\pr{n}\leq Cn^{-1-2/\delta-\eta}$. When $\delta<2$, our condition is less restrictive but our approach does not give better a weaker condition than \ref{assum:mixing_TLFC} when we assume moments of order higher than two. 
The same authors obtained in \cite{MR3282026} a central limit theorem for Hilbert-valued 
$U$-statistics whose date comes from functionals of an absolutely regular sequence. 
In \cite{MR1851171}, the more general case of functions of absolutely regular processes has been addressed but when restricted to $\beta$-mixing case, the obtained result is not better than ours.

Finally, let us mention that Theorem~\ref{thm:TLFC} may be applied even if there is no positive $\delta$ for which $\E{\norm{h_1\pr{X_0}}_{\Hi}^{2+\delta}}$ is finite. For instance, if we merely have $\E{\norm{h_1\pr{X_0}}_{\Hi}^2\pr{\log\pr{1+\norm{h_1\pr{X_0}}_{\Hi}} }^{\gamma}}<\infty$, then \eqref{eq:cond_melange_TLCF} 
may be satisfied, but at the cost of the existence of constants $C$ and  $a\in \pr{0,1}$ such that $\alpha\pr{\sigma\pr{X_i,i\leq 0},\sigma\pr{X_k}}\leq C a^k $. We refer the reader to \cite{MR2117923}, pages 155-158.

\subsection{Marcinkiewicz strong law of large numbers} \label{subsec:SLLN}

In all the results on the Marcinkiewicz strong law of large numbers, a supplementary moment condition with respect to the i.i.d.\ case will be required. 
For $1<p<2$ and $\delta>0$, we will consider the assumption 
\begin{equation}\label{eq:moment_assumption_p_plus_delta}
\E{\norm{h\pr{X_1,X'_1}}_{\Hi}^{p+\delta}}<\infty,
\end{equation}
where $X'_1$ is an independent copy of $X_1$.
Since the derivation of moment inequalities was done via the use of moment inequalities for martingale, it will be natural to distinguish between the cases where the exponent $p+\delta$ is bigger than $2$ or not. 
An other reasonable assumption is boundedness in $\mathbb L^{1}$ of the summands which compose a $U$-statistic, namely, 
\begin{equation}\label{eq:boundedness_in_L1}
\sup_{j\geq 2}\E{\norm{h\pr{X_1,X_j}}_{\Hi}}<\infty.
\end{equation}

\begin{Theorem}[Law of large numbers, non-degenerated case, $p+\delta\geq 2$]\label{thm:MSLLN_non_deg_finite_var}
Let $\pr{X_i}_{i\in\Z}$ be a strictly stationary sequence of random variables taking values in a separable metric space $\pr{S,d}$ and let $h\colon S\times S\to \Hi$ be a measurable function, where $S\times S$ is endowed with the product $\sigma$-algebra and the Hilbert space $\Hi$ with the $\sigma$-algebra induced by the norm. Let $1<p<2$. 
Suppose that \eqref{eq:moment_assumption_p_plus_delta} and \eqref{eq:boundedness_in_L1} hold with $p+\delta\geq 2$  and  the following condition  is  satisfied:
  there exists a positive $\eta$ such that 
\begin{equation}\label{assump:mixing_linear_part_LGN}
\sum_{k=1}^\infty k^{\gamma\pr{p,\delta,\eta}}<\infty,
\end{equation}
where 
\begin{equation}
\gamma\pr{p,\delta,\eta}=\max\ens{p-1+\eta,p-2+ \frac{p\pr{p-1}}\delta}.
\end{equation}
Then the following convergence takes place:
\begin{equation}
\lim_{n\to\infty}\frac 1{n^{1+1/p}}\norm{U_n\pr{h}}_{\Hi}=0\mbox{ a.s..}
\end{equation}
\end{Theorem}

\begin{Theorem}[Law of large numbers, non-degenerated case, $p+\delta< 2$]\label{thm:MSLLN_non_deg_infinite_var}
Let $\pr{X_i}_{i\in\Z}$ be a strictly stationary sequence of random variables taking values in a separable metric space $\pr{S,d}$ and let $h\colon S\times S\to \Hi$ be a measurable function, where $S\times S$ is endowed with the product $\sigma$-algebra and the Hilbert space $\Hi$ with the $\sigma$-algebra induced by the norm. Let $1<p<2$. 
Suppose that \eqref{eq:moment_assumption_p_plus_delta} and \eqref{eq:boundedness_in_L1} hold with $p+\delta< 2$ and   
that 
\begin{equation}\label{eq:cond_melange_non_deg_infinite_var}
\sum_{k=1}^\infty k^{\gamma\pr{p,\delta}}\beta\pr{k}<\infty,
\end{equation}
where 
\begin{equation}
\gamma\pr{p,\delta}=\max\ens{p-2+\frac{p\pr{p-1}}\delta, 
\frac{p\pr{p-1}+\pr{p-1}\delta}{p\pr{p-1}+\pr{p+1}\delta}}
\end{equation}
Then the following convergence takes place:
\begin{equation}\label{eq:MLLN_non_deg}
\lim_{n\to\infty}\frac 1{n^{1+1/p}}\norm{U_n\pr{h}}_{\Hi}=0\mbox{ a.s..}
\end{equation}
\end{Theorem}

Dehling and Sharipov \cite{MR2571765} obtained also the convergence 
\eqref{eq:MLLN_non_deg}, but with  slightly different assumptions on the $\beta$-mixing 
coefficients, namely, 
\begin{equation}
\sum_{k=1}^\infty k^{\gamma'\pr{p,\delta}}\beta\pr{k}<\infty, 
\end{equation}
where 
\begin{equation}
\gamma'\pr{p,\delta}=\max\ens{p-2+\frac{p\pr{p-1}}\delta, 1}.
\end{equation}
Since $\gamma'\pr{p,\delta}\geq \max\ens{\gamma\pr{p,\delta},\gamma\pr{p,\delta,\eta}}$, our assumption is always equaly or less restrictive. It is also worth pointing out that we do not need symmetry of the kernel. Moreover, we 
can also treat Hilbert space valued kernels. These extensions allows us to consider kernels of the form 
\begin{equation}
h\colon \Hi\times\Hi\to\Hi, \quad h\pr{x,y}=\begin{cases}
\frac{x-y}{\norm{x-y}_{\Hi}}&\mbox{ if }x\neq y,\\
0&\mbox{ if }x=y.,
\end{cases}
\end{equation}
which played an important role in \cite{MR3650400,wegner2023robust}. 

We now continue the presentation of the results in order to address the degenerated case. 
Like in the independent case, the appropriated normalisation is $n^{2/p}$.

\begin{Theorem}[Law of large numbers, degenerated case, $p+\delta= 2$]\label{thm:MSLLN_deg_finite_var}
Let $\pr{X_i}_{i\in\Z}$ be a strictly stationary sequence of random variables taking values in a separable metric space $\pr{S,d}$ and let $h\colon S\times S\to \Hi$ be a   measurable function, where $S\times S$ is endowed with the product $\sigma$-algebra and the Hilbert space $\Hi$ with the $\sigma$-algebra induced by the norm.  
Suppose that 
\begin{equation}
\E{h\pr{X_1,X'_1}\mid X_1}=\E{h\pr{X_1,X'_1}\mid X'_1}=0.
\end{equation}
Let $1<p<2$. 
Assume that \eqref{eq:moment_assumption_p_plus_delta} and \eqref{eq:boundedness_in_L1} hold with $p+\delta\geq 2$ and that there exists some positive $\eta$ such that 
\begin{equation}\label{eq:assum_mix_MSLLN_deg_finite_var}
\sum_{k=1}^\infty k^{\frac{2\pr{p-1}}{2-p}+\eta}\beta\pr{k}<\infty.
\end{equation}
Then the following convergence takes place:
\begin{equation}
\lim_{n\to\infty}\frac 1{n^{2/p}}\norm{U_n\pr{h}}_{\Hi}=0\mbox{ a.s..}
\end{equation}
\end{Theorem}

Note that in the previous result, we only assume the existence of moments of order two. In our approach, moments of higher order will not help to find a weaker condition on the decay of the $\beta$-mixing coefficients. 

\begin{Theorem}[Law of large numbers, degenerated case, $p+\delta< 2$]\label{thm:MSLLN_deg_infinite_var}
Let $\pr{X_i}_{i\in\Z}$ be a strictly stationary sequence of random variables taking values in a separable metric space $\pr{S,d}$ and let $h\colon S\times S\to \Hi$ be a   measurable function, where $S\times S$ is endowed with the product $\sigma$-algebra and the Hilbert space $\Hi$ with the $\sigma$-algebra induced by the norm. Suppose that 
\begin{equation}
\E{h\pr{X_1,X'_1}\mid X_1}=\E{h\pr{X_1,X'_1}\mid X'_1}=0.
\end{equation} Let $1<p<2$. 
Assume that \eqref{eq:moment_assumption_p_plus_delta} and \eqref{eq:boundedness_in_L1} hold with $p+\delta< 2$ and  
\begin{equation}\label{eq:assum_mix_MSLLN_deg_infinite_var}
\sum_{k=1}^\infty k^{ p-1+\frac{p\pr{p-1}}{\delta}}\beta\pr{k}<\infty.
\end{equation}
Then the following convergence takes place:
\begin{equation}
\lim_{n\to\infty}\frac 1{n^{2/p}}\norm{U_n\pr{h}}_{\Hi}=0\mbox{ a.s..}
\end{equation}
\end{Theorem}
These two theorems complement the ones obtained in \cite{MR2571765}, where 
the degenerated case was considered, but only in the case of a bounded real valued kernel.  

\section{Proofs}\label{sec:proof}

\subsection{A general deviation inequality}\label{subsec:dev_ineg}

In this subsection, we give a bound for the maximum of a $U$-statistic of strictly stationary data in terms of $U$-statistics 
of i.i.d.\ data and partial sums of sequences of random variables.

\begin{Proposition}\label{prop:ineg_dev_Ustat_degeneree}
Let  $N>2q\geq 1$ be integers, let $R,x>0$, let $\pr{X_i}_{i\geq 1}$ be a strictly stationary sequence of random variables taking values in a separable metric space $\pr{S,d}$ and let  $h\colon S^2\to\Hi$ be a measurable function, where $\pr{\Hi,\scal{}{}}$ is a separable Hilbert space,
such that 
\begin{equation}\label{eq:noyau_degenere}
\E{h\pr{X_1,X'_1}\mid X_1}=\E{h\pr{X_1,X'_1}\mid X'_1}=0,
\end{equation}
  where $X'_1$ be an independent copy of $X_1$. 
Define $H:=  \norm{h\pr{X_1,X'_1}}_{\Hi}$.
  For each $r\geq 2$, the following inequality takes place
 \begin{multline}\label{eq:ineg_dev_Ustat_degeneree}
 \PP\pr{\max_{2\leq n\leq N}\norm{\sum_{1\leq i<j\leq n}h\pr{X_i,X_j}}_{\Hi}>x}\leq C_r x^{-r}q^{r}N^r\E{H^r\ind{H\leq R}}\\
 + C_rx^{-1}N^2 \E{H\ind{H>R}}  
 +C_r x^{-1}qN\sup_{j\geq 2}\E{\norm{h\pr{X_1,X_j}}_{\Hi}}+4N\beta\pr{q},
 \end{multline}
 where the constant $C_r$ depends only on $r$.
\end{Proposition}
Note that when $H$ has a finite moment of order $r$, one can simply let $R$ going to infinity, which gives the simpler inequality 
 \begin{multline}\label{eq:ineg_dev_Ustat_degeneree_moment_ordre_2}
 \PP\pr{\max_{2\leq n\leq N}\norm{\sum_{1\leq i<j\leq n}h\pr{X_i,X_j}}_{\Hi}>x}\leq C_rx^{-r}q^{r}N^r\E{H^r }\\
 +C_r x^{-1}qN\sup_{j\geq 2}\E{\norm{h\pr{X_1,X_j}}_{\Hi}}+4N\beta\pr{q}.
 \end{multline}
The contribution of  each term is either increasing or decreasing in $q$ hence $q$ has to be chosen in a judicious way.

We start the proof of Proposition~\ref{prop:general_deviation_inequality} 
by a Lemma which gives a bound of the maximum of a $U$-statistics in terms of 
several $U$-statistics whose data can be expressed in terms of blocks 
of  vectors of elements of $\pr{X_i}_{i\geq 1}$ having a gap of size $q$. 
\begin{Lemma}\label{lem:borne_U_stat}
Let $N>2q\geq 1$ be integers. Let $\pr{X_i}_{i\geq 1}$ be a   sequence of random variables taking values in a separable metric space $\pr{S,d}$ and let $h\colon S^2\to\Hi$ be a measurable function, where $\pr{\Hi,\scal{}{}}$ is a separable Hilbert space. Define the $S^{q-1}$-valued vector $V_{k,u}$ by 
\begin{equation}\label{eq:def_Vku}
V_{k,u}=\pr{X_{2qu+k},\dots,X_{2qu+k+q-1}}.
\end{equation}
 Then the following inequality holds:
\begin{equation}\label{eq:bound_Ustat}
\max_{2\leq n\leq N}\norm{\sum_{1\leq i<j\leq n}h\pr{X_i,X_j}}_{\Hi}
\leq \sum_{i=1}^4 M_{N,q,i}+\sum_{i=1}^5R_{N,q,i},
\end{equation}
where 
\begin{align}\label{eq:def_MNq}
M_{N,q,1}&=\sum_{\substack{\ell,\ell'=1\\ 0\leq \ell-\ell'\leq q-1}}^{2q}
\max_{1\leq m\leq \ent{\frac{N}{2q}}}
\norm{\sum_{0\leq u<v\leq m}h_{\ell,\ell'}\pr{V_{\ell',u},V_{\ell',v}   }}_{\Hi},\\
M_{N,q,2}&= \sum_{\substack{\ell,\ell'=1\\ 0\leq \ell'-\ell\leq q-1}}^{2q}
\max_{1\leq m\leq \ent{\frac{N}{2q}}}
\norm{\sum_{0\leq u<v\leq m}h_{\ell,\ell'}\pr{V_{\ell,u},V_{\ell,v}   }}_{\Hi},\\
M_{N,q,3}&= \sum_{\substack{\ell,\ell'=1\\ q\leq \ell-\ell'\leq 2q-1}}^{2q}
\max_{1\leq m\leq \ent{\frac{N}{2q}}}
\norm{\sum_{0\leq u<v\leq m}h_{\ell,\ell'}\pr{V_{\ell-2q,u},V_{\ell-2q,v}   }}_{\Hi},\\
M_{N,q,4}&= \sum_{\substack{\ell,\ell'=1\\ q\leq \ell'-\ell\leq 2q-1}}^{2q}
\max_{1\leq m\leq \ent{\frac{N}{2q}}}
\norm{\sum_{0\leq u<v\leq m}h_{\ell,\ell'}\pr{V_{\ell',u},V_{\ell',v}   }}_{\Hi},
\end{align}
the functions $h_{\ell,\ell'}\colon S^{q}\times S^{q}\to\Hi$ are defined by 
\begin{equation}\label{eq:def_h_ell_ell'}
h_{\ell,\ell'}\pr{\pr{x_a}_{a=1}^q,\pr{y_b}_{b=1}^q }
= \begin{cases}
h\pr{x_{\ell-\ell'+1},y_{ 1}}&\mbox{if }0\leq \ell-\ell'\leq q-1,\\
h\pr{x_{1},y_{\ell'+1}}&\mbox{if }0\leq \ell'-\ell\leq q-1,\\
h\pr{x_{1},y_{\ell'-\ell+2q+1}}&\mbox{if }q\leq \ell-\ell'\leq 2q-1,
\\
h\pr{x_{\ell'-\ell+2q+1},y_{1}}&\mbox{if }q\leq \ell'-\ell\leq 2q-1,\\
\end{cases}
\end{equation}
\begin{align}
R_{N,q,1}&=\sum_{\substack{\ell,\ell'=1\\ q\leq \ell-\ell'\leq 2q-1}}^{2q}\max_{1\leq m\leq \ent{\frac{N}{2q}}}
\norm{\sum_{v=1}^m h\pr{X_{2\pr{v-1}q+\ell},X_{2vq+\ell'}} }_{\Hi},\label{eq:def_RNq1}\\
R_{N,q,2}&=\sum_{\substack{\ell,\ell'=1\\ q\leq \ell-\ell'\leq 2q-1}}^{2q}\max_{1\leq m\leq \ent{\frac{N}{2q}}}
\norm{\sum_{v=1}^m h\pr{X_{ \ell},X_{2vq+\ell'}} }_{\Hi},
\label{eq:def_RNq2}\\
R_{N,q,3}&=\sum_{\substack{\ell,\ell'=1\\ q\leq \ell'-\ell\leq 2q-1}}^{2q}\max_{1\leq m\leq \ent{\frac{N}{2q}}}\norm{\sum_{v=1}^m h\pr{X_{ \ell},X_{2qv+\ell'}} }_{\Hi},\label{eq:def_RNq3}\\
R_{N,q,4}&=\sum_{\substack{\ell,\ell'=1\\ q\leq \ell'-\ell\leq 2q-1}}^{2q}\max_{1\leq m\leq \ent{\frac{N}{2q}}}
\norm{\sum_{v=0}^m h\pr{X_{ 2qv+\ell},X_{2qv+\ell'}} }_{\Hi}\label{eq:def_RNq4}
,\\
R_{N,q,5}&=\sum_{1\leq \ell<\ell'\leq 2q}
\max_{2\leq m\leq\ent{\frac{N}{2q}}}\norm{\sum_{u=0}^m 
h\pr{X_{2qu+\ell,2qv+\ell'}}
}_{\Hi}\label{eq:def_RNq5}.
\end{align}
\end{Lemma}

Let us explain the roled played by the different terms involved in \eqref{eq:bound_Ustat}.
The four terms $M_{N,q,i}$, $i\in\ens{1,2,3,4}$ are maxima of $U$-statistics whose data may be dependent, but their behavior is close to the one of $U$-statistics of ndependent data. The other terms $R_{N,q,i}$, $i\in\ens{1,2,3,4,5}$, play a negligible role. Indeed, their contribution is of the same order as that of maximum of partial sums of a strictly stationary sequence, but with the normalization corresponding to a $U$-statistic, which is stronger than the one needed.

\begin{Proposition}\label{prop:general_deviation_inequality}
Let $N>2q\geq 1$ be integers. Let $\pr{X_i}_{i\geq 1}$ be a strictly stationary sequence of random variables taking values in a separable metric space $\pr{S,d}$ and let $h\colon S^2\to\Hi$ be a measurable function, where $\pr{\Hi,\scal{}{}}$ is a separable Hilbert space. 
 Then the following inequality holds:
\begin{multline}
\PP\pr{\max_{2\leq n\leq N}\norm{\sum_{1\leq i<j\leq n}h\pr{X_i,X_j}}_{\Hi}>8x}\leq 
\sum_{i=1}^4\PP\pr{M_{N,q,i}^*>x}
    +\sum_{i=1}^5\PP\pr{R_{N,q,i}>x }+4N\beta\pr{q},
\end{multline}
where the terms $R_{N,q,i}$ are defined as in \eqref{eq:def_RNq1}, \eqref{eq:def_RNq2}, \eqref{eq:def_RNq3}, \eqref{eq:def_RNq4} and \eqref{eq:def_RNq5},
\begin{align}\label{eq:def_MNq*}
M^*_{N,q,1}&=\sum_{\substack{\ell,\ell'=1\\ 0\leq \ell-\ell'\leq q-1}}^{2q}
\max_{1\leq m\leq \ent{\frac{N}{2q}}}
\norm{\sum_{0\leq u<v\leq m}h_{\ell,\ell'}\pr{V^*_{\ell',u},V^*_{\ell',v}   }}_{\Hi},\\
M^*_{N,q,2}&= \sum_{\substack{\ell,\ell'=1\\ 0\leq \ell'-\ell\leq q-1}}^{2q}
\max_{1\leq m\leq \ent{\frac{N}{2q}}}
\norm{\sum_{0\leq u<v\leq m}h_{\ell,\ell'}\pr{V^*_{\ell,u},V^*_{\ell,v}   }}_{\Hi},\\
M^*_{N,q,3}&= \sum_{\substack{\ell,\ell'=1\\ q\leq \ell-\ell'\leq 2q-1}}^{2q}
\max_{1\leq m\leq \ent{\frac{N}{2q}}}
\norm{\sum_{0\leq u<v\leq m}h_{\ell,\ell'}\pr{V^*_{\ell-2q,u},V^*_{\ell-2q,v}   }}_{\Hi},\\
M^*_{N,q,4}&= \sum_{\substack{\ell,\ell'=1\\ q\leq \ell'-\ell\leq 2q-1}}^{2q}
\max_{1\leq m\leq \ent{\frac{N}{2q}}}
\norm{\sum_{0\leq u<v\leq m}h_{\ell,\ell'}\pr{V^*_{\ell',u},V^*_{\ell',v}   }}_{\Hi},
\end{align}
and $\pr{V^*_{k,u}}_{\substack{k\in\ens{-q,\dots,q}\\ 
u\in\ens{0,\dots,\ent{N/\pr{2q}}} }}$ satisfy the following:
\begin{equation}\label{eq:indep_coupled_vectors}
\mbox{ for each } k\in\ens{-q,\dots,q}, 
\pr{V^*_{k,u}}_{u=0}^{\ent{N/{2q}}}\mbox{ is independent},
\end{equation}
\begin{equation}\label{eq:id_distr_coupled_vectors}
\mbox{ for each } k\in\ens{-q,\dots,q}, u\in\ens{0,\dots,\ent{N/\pr{2q}}}, 
V_{k,u}^*\mbox{ has the same distribution as }V_{k,u}.
\end{equation}
\end{Proposition}

It will turn out that we will not need the joint distribution of 
$M_{N,q,i}^*$
 Only the fact that all there random variables are identically distributed and that the common distribution is the same as the maximum of norm of a $U$-statistic having kernel $h$ and i.i.d.\ data (with the same distribution as $X_1$)  will 
play a decisive role.

\begin{proof}[Proof of Lemma~\ref{lem:borne_U_stat}]
We express for a fixed $n$ the set $I:=\ens{\pr{i,j}\in\N^2,1\leq i<j\leq n}$ 
as 
\begin{equation}
I=\bigcup_{\ell,\ell'=1}^{2q}\ens{\pr{2qu+\ell,2qv+\ell'},0\leq u\leq v\leq \ent{\pr{n-\ell'}/\pr{2q}}, 2qu+\ell<2qv+\ell'},
\end{equation}
which can be splitted as the disjoint union $I=\bigcup_{a=1}^5I_{n,a}$, where 
\begin{align}
I_{n,1}&=\bigcup_{\substack{\ell,\ell'=1 \\  0\leq \ell-\ell'\leq q-1  }}^{2q}I_{n,1,\ell,\ell'}, \quad I_{n,1,\ell,\ell'}=\ens{\pr{2qu+\ell,2qv+\ell'}, 0\leq u< v\leq \ent{\pr{n-\ell'}/\pr{2q}}},\\
I_{n,2}&=\bigcup_{\substack{\ell,\ell'=1 \\  1\leq \ell'-\ell\leq q-1  }}^{2q}
I_{n,2,\ell,\ell'},\quad I_{n,2,\ell,\ell'}=
 \ens{\pr{2qu+\ell,2qv+\ell'},  0\leq u< v\leq \ent{\pr{n-\ell'}/\pr{2q}}},\\
I_{n,3}&=\bigcup_{\substack{\ell,\ell'=1 \\  q\leq \ell-\ell'\leq 2q-1 }}^{2q}
I_{n,3,\ell,\ell'},\quad I_{n,3,\ell,\ell'}:=
\ens{\pr{2qu+\ell,2qv+\ell'},0\leq u< v\leq \ent{\pr{n-\ell'}/\pr{2q}}},\\
I_{n,4}&=\bigcup_{\substack{\ell,\ell'=1 \\  q\leq \ell'-\ell\leq 2q-1 }}^{2q}I_{n,4,\ell,\ell'}, \quad
I_{n,4,\ell,\ell'}=\ens{\pr{2qu+\ell,2qv+\ell'}, 0\leq u< v\leq \ent{\pr{n-\ell'}/\pr{2q}}},\\
I_{n,5}&=\bigcup_{1\leq \ell<\ell'\leq 2q}I_{n,5,\ell,\ell'},\quad I_{n,5,\ell,\ell'}=\ens{\pr{2qu+\ell,2qu+\ell'},0\leq u\leq\ent{\pr{n-\ell'}/\pr{2q}}}.
\end{align}
The first two sets $I_1$ and $I_2$ contain the indices of the form $\pr{2qu+\ell,2qv+\ell'}$ for which $\abs{\ell-\ell'}\leq q-1$ and a distinction is made according to the order between $\ell$ and $\ell'$. The sets $I_3$ and $I_4$ contain the  indices of the form $\pr{2qu+\ell,2qv+\ell'}$ for which $\abs{\ell-\ell'}> q-1$ (since $1\leq \ell,\ell'\leq 2q$, we necessarily have $\abs{\ell-\ell'}\leq 2q-1$)  and here again, a distinction is made according to the order between $\ell$ and $\ell'$. Note that in the sets $I_a$, $1\leq a\leq 4$, one has $u<v$ which guarantees that $2qu+\ell<2qv+\ell'$.
Finally, in the set $I_5$, the indexes corresponding to the case $u=v$ are collected.

As a consequence, the following inequality takes place:
\begin{multline}\label{eq:decomposition_de_U_n_1ere_etape}
\max_{2\leq n\leq N}\norm{\sum_{1\leq i<j\leq n}h\pr{X_i,X_j}}_{\Hi}
\leq \sum_{\substack{\ell,\ell'=1 \\  0\leq \ell-\ell'\leq q-1  }}^{2q}\max_{2\leq n\leq N}
\norm{\sum_{0\leq u<v\leq\ent{\frac{n-\ell'}{2q}}} h\pr{X_{2qu+\ell},X_{2qv+\ell'}}    }_{\Hi}\\
+\sum_{\substack{\ell,\ell'=1 \\  1\leq \ell'-\ell\leq q-1  }}^{2q}
\max_{2\leq n\leq N}\norm{\sum_{0\leq u<v\leq\ent{\frac{n-b}{2q}}} h\pr{X_{2qu+\ell},X_{2qv+\ell'}}    }_{\Hi}\\+
\sum_{\substack{\ell,\ell'=1 \\  q\leq \ell-\ell'\leq 2q-1 }}^{2q}
\max_{2\leq n\leq N}\norm{\sum_{0\leq u< v\leq \ent{\pr{n-\ell'}/\pr{2q}}} 
h\pr{X_{2qu+\ell},X_{2qv+\ell'}}
   }_{\Hi}\\
   +\sum_{\substack{\ell,\ell'=1 \\  q\leq \ell'-\ell\leq 2q-1 }}^{2q}
   \max_{2\leq n\leq N}\norm{\sum_{0\leq u< v\leq \ent{\pr{n-\ell'}/\pr{2q}}} 
h\pr{X_{2qu+\ell},X_{2qv+\ell'}}}_{\Hi}\\+
\sum_{1\leq \ell<\ell'\leq 2q}\max_{2\leq n\leq N}
\norm{\sum_{u=0}^{\ent{\pr{n-\ell'}/\pr{2q}}}  h\pr{X_{2qu+\ell},X_{2qu+\ell'}}  }_{\Hi}.
\end{multline}
By the elementary inequality
\begin{equation}\label{eq:ineg_maximum_suite_nombre}
\max_{2\leq n\leq N}a_{\ent{\pr{n-\ell'}/\pr{2q}}}
\leq\max_{0\leq m\leq \ent{N/\pr{2q}}} a_m
\end{equation}
valid for a non-negative sequence $\pr{a_m}_{m\geq 0}$, we derive from 
\eqref{eq:decomposition_de_U_n_1ere_etape} that 
\begin{multline}\label{eq:decomposition_de_U_n_2eme_etape}
\max_{2\leq n\leq N}\norm{\sum_{1\leq i<j\leq n}h\pr{X_i,X_j}}_{\Hi}
\leq \sum_{\substack{\ell,\ell'=1 \\  0\leq \ell-\ell'\leq q-1  }}^{2q}\max_{0\leq m\leq \ent{\frac{N}{2q}}}
\norm{\sum_{0\leq u<v\leq m} h\pr{X_{2qu+\ell},X_{2qv+\ell'}}    }_{\Hi}\\
+\sum_{\substack{\ell,\ell'=1 \\  1\leq \ell'-\ell\leq q-1  }}^{2q}
\max_{0\leq m\leq \ent{\frac{N}{2q}}}\norm{\sum_{0\leq u<v\leq m} h\pr{X_{2qu+\ell},X_{2qv+\ell'}}    }_{\Hi}+
\sum_{\substack{\ell,\ell'=1 \\  q\leq \ell-\ell'\leq 2q-1 }}^{2q}
\max_{0\leq m\leq \ent{\frac{N}{2q}}}\norm{\sum_{0\leq u< v\leq  m} 
h\pr{X_{2qu+\ell},X_{2qv+\ell'}}
   }_{\Hi}\\
   +\sum_{\substack{\ell,\ell'=1 \\  q\leq \ell'-\ell\leq 2q-1 }}^{2q}
\max_{0\leq m\leq \ent{\frac{N}{2q}}}\norm{\sum_{0\leq u< v\leq m} 
h\pr{X_{2qu+\ell},X_{2qv+\ell'}}}_{\Hi}\\+
\sum_{1\leq \ell<\ell'\leq 2q}\max_{0\leq m\leq \ent{\frac{N}{2q}}}
\norm{\sum_{u=0}^{m}  h\pr{X_{2qu+\ell},X_{2qu+\ell'}}  }_{\Hi}=:A_1+A_2+A_3+A_4+A_5.
\end{multline}
Using the expression of 
$h_{\ell,\ell'}$  and $V_{k_,u}$ given by \eqref{eq:def_h_ell_ell'} and \eqref{eq:def_Vku},  we bound $A_1$ by $M_{N,q,1}$ and $A_2$ by $M_{N,q,2}$. Moreover, $A_5$ coincides with $R_{N,q,5}$. For $A_3$, we write 
\begin{multline}
\sum_{0\leq u< v\leq m} 
h\pr{X_{2qu+\ell},X_{2qv+\ell'}}
=\sum_{0\leq u< v\leq m} 
h\pr{X_{2q\pr{u-1}+\ell},X_{2qv+\ell'}}\\
+\sum_{v=1}^m
\sum_{u=0}^{v-1}\pr{h\pr{X_{2qu+\ell},X_{2qv+\ell'}}-h\pr{X_{2q\pr{u-1}+\ell},X_{2qv+\ell'}}},
\end{multline}
and since the sum over $u$ is telescopic, the inequality 
\begin{multline}\label{eq:borne_A3}
\norm{\sum_{0\leq u< v\leq m} 
h\pr{X_{2qu+\ell},X_{2qv+\ell'}}}_{\Hi}
\leq \norm{\sum_{0\leq u< v\leq m} 
h\pr{X_{2q\pr{u-1}+\ell},X_{2qv+\ell'}}}_{\Hi}
\\
+ \norm{\sum_{v=1}^mh\pr{X_{2qv+\ell},X_{2qv+\ell'}}}_{\Hi}
+\norm{\sum_{v=1}^mh\pr{X_{\ell},X_{2qv+\ell'}}}_{\Hi}
\end{multline}
takes place.
Then we use the expression of $h_{\ell,\ell'}$ in order to show that $A_3\leq M_{N,q,3}+
R_{N,q,1}+R_{N,q,3}$.
The treatment of $A_4$ is similar, with the minor difference that we use the decomposition
\begin{multline}
\sum_{0\leq u< v\leq m} 
h\pr{X_{2qu+\ell},X_{2qv+\ell'}}
=\sum_{0\leq u< v\leq m} 
h\pr{X_{2q\pr{u+1}+\ell},X_{2qv+\ell'}}\\
+\sum_{v=1}^m
\sum_{u=0}^{v-1}\pr{h\pr{X_{2qu+\ell},X_{2qv+\ell'}}-h\pr{X_{2q\pr{u+1}+\ell},X_{2qv+\ell'}}}.
\end{multline}
This ends the proof of Lemma~\ref{lem:borne_U_stat}.
\end{proof}

In order to prove Proposition~\ref{prop:general_deviation_inequality}, we will 
need the following coupling lemma, due to Berbee \cite{MR0547109}.
\begin{Lemma}\label{lem:coupling}
Let $X$ and $Y$ be random variables defined on $\pr{\Omega,\Fca,\PP}$ with 
values in a Polish space $S$. Let $\sigma\pr{X}$ be the $\sigma$-algebra generated by $X$ and let $U$ be a random variable uniformly distributed on $[0,1]$ and independent 
of $\pr{X,Y}$. There exists a random variable $Y^*$, measurable with respect to 
$\sigma\pr{X}\vee \sigma\pr{Y}\vee \sigma\pr{U}$, independent of $X$ and distributed as $Y$, and such that $\PP\pr{Y\neq Y^*}=\beta\pr{X,Y}$.
\end{Lemma}

\begin{proof}[Proof of Proposition~\ref{prop:general_deviation_inequality}]
A consequence of Lemma~\ref{lem:coupling} is the following. Given a sequence $\pr{Y_u}_{u\geq 1}$ defined on a probability space $\pr{\Omega,\Fca,\PP}$ with values in a Polish space $S$, we can find, on a richer probability space, an indepdent sequence $\pr{Y_u^*}_{u\geq 1}$ of random variable, such that for each $u\geq 1$, $Y_u$ have the same distribution as $Y_u^*$ 
and $\PP\pr{Y^*_u\neq Y_u}\leq \beta\pr{\sigma\pr{Y_i,i\leq u-1 },\sigma\pr{Y_u}}$.

Applying this result for each fixed $k\in\ens{-q,\dots,q}$ gives sequences 
$\pr{V^*_{k,u}}_{u=0}^{\ent{N/\pr{2q}}}$ satisfying \eqref{eq:indep_coupled_vectors}, \eqref{eq:id_distr_coupled_vectors} and  $\PP\pr{V^*_{k,u}\neq V_{k,u}}\leq 
\beta\pr{\sigma\pr{V_{k,i},i\leq u-1},\sigma\pr{V_{k,i}}}$.
 Define the events 
\begin{equation}
A_{k,u}:=\ens{V^*_{k,u}\neq V_{k,u}},\quad k\in\ens{-q,\dots,q}, u\in\ens{0,\dots,\ent{N/\pr{2q}}}
\end{equation}
and 
\begin{equation}
A=\bigcup_{k=-q}^q\bigcup_{u=0}^{\ent{N/\pr{2q}}}A_{k,u}.
\end{equation}
In view of Lemma~\ref{lem:borne_U_stat}, one has 
\begin{multline}
\PP\pr{\max_{2\leq n\leq N}\norm{\sum_{1\leq i<j\leq n}h\pr{X_i,X_j}}_{\Hi}>5x}\leq 
\PP\pr{\ens{M_{N,q,1}+M_{N,q,2}+M_{N,q,3}+M_{N,q,4}>x}\cap A^c  }\\
+\PP\pr{A}+\sum_{i=1}^4\PP\pr{M_{N,q,i}>x}.
\end{multline}
Since the vectors $V_{k,u}$ coincide with the vectors $V^*_{k,u}$, $k\in\ens{-q,\dots,q}, u\in\ens{0,\dots,\ent{N/\pr{2q}}}$ on $A^c$, the events $\ens{M_{N,q,1}+M_{N,q,2}+M_{N,q,3}+M_{N,q,4}>x}\cap A^c$ and $\ens{M_{N,q,1}^*+M_{N,q,2}^*+M_{N,q,3}^*+M_{N,q,4}^*>x}\cap A^c$ are equal, where 
 $M_{N,q,i}^*$, $i\in\ens{1,2,3,4}$, are defined by \eqref{eq:def_MNq*}. 
 
Moreover, since $A$ is a union of $\pr{2q+1}\pr{\ent{N/\pr{2q}}+1}$
sets having probability smaller than $\beta\pr{q}$, one derives that 
$\PP\pr{A}\leq \pr{2q+1}\pr{\ent{N/\pr{2q}}+1}\beta\pr{q}
\leq 4N\beta\pr{q}$. This ends the proof of Proposition~\ref{prop:general_deviation_inequality}.
\end{proof}

\begin{proof}[Proof of Proposition~\ref{prop:ineg_dev_Ustat_degeneree}]
We apply Proposition~\ref{prop:general_deviation_inequality} with $x$ replaced by $2x$ and a use of Markov's inequality gives 
\begin{multline}\label{eq:bound_deg_Ustat_demo}
\PP\pr{\max_{2\leq n\leq N}\norm{\sum_{1\leq i<j\leq n}h\pr{X_i,X_j}}_{\Hi}>16x}\\ \leq 
\sum_{i=1}^4\PP\pr{M_{N,q,i}^*>2x}
    +\pr{2x}^{-1}\sum_{i=1}^5\E{R_{N,q,i}}+4N\beta\pr{q}.
\end{multline}
One has 
\begin{equation}
R_{N,q,i}\leq \sum_{\ell,\ell'=1}^{2q}\sum_{v=0}^{\ent{\frac{N}{2q}}}
\norm{Y_{\ell,\ell',v}}_{\Hi},
\end{equation}
where $Y_{\ell,\ell',v}$ equal $h\pr{X_i,X_j}$ for some indices $i$ and $j$.  Therefore, using $\ent{N/\pr{2q}}+1\leq N/q$, we find
\begin{equation}
\E{R_{N,q,i}}\leq  4qN\sup_{j\geq 2}\E{\norm{h\pr{X_1,X_j}}_{\Hi}}.
\end{equation}

In order to control the terms $M_{N,q,i}$, we introduce truncated and degenerated kernels as follows
For a  kernel $h\colon S^2\to\Hi$ and a strictly stationary sequence $\pr{X_i}_{i\geq 1}$, define 
\begin{equation}\label{eq:def_de_h_deg}
h^{\deg}\pr{x,y}=h\pr{x,y}-\E{h_1\pr{X_1,y}}-\E{h_1\pr{x,X_1}}+\E{h\pr{X_1,X'_1}},
\end{equation}
where $X'_1$ is an independent copy of $X_1$. In this way, one has 
\begin{equation}\label{eq:propriete_deg_de_h_deg}
\E{h^{\deg}\pr{X_1,X'_1}\mid X_1}=\E{h^{\deg}\pr{X_1,X'_1}\mid X'_1}=0
\end{equation}
hence if $\pr{\xi_i}_{i\geq 1}$ is an i.i.d.\ sequence and $\xi_1$ has the same distribution as $X_1$, the $U$-statistic of data $\pr{\xi_i}_{i\geq 1}$ and kernel $h^{\deg}$ is degenerated. Note that this property of degeneracy depends on the law of $X_1$, but since there will be no ambiguity, we will not write this dependence. Let $1<r\leq 2$, $N\geq 2q\geq 1$ and $R>0$.
Let $h_{\leq}\colon S^2\to\Hi$ and $h_{>}\colon S^2\to\Hi$ 
be defined as 
\begin{equation}
h_{\leq}\pr{x,y}=h\pr{x,y}\ind{\norm{h\pr{x,y}}_{\Hi}\leq R}
\end{equation}
\begin{equation}
h_{>}\pr{x,y}=h\pr{x,y}\ind{\norm{h\pr{x,y}}_{\Hi}> R}.
\end{equation}
In view of \eqref{eq:noyau_degenere}, one has 
$h\pr{x,y}=\pr{h_{\leq}}^{\deg}\pr{x,y}+\pr{h_{>}}^{\deg}\pr{x,y}$. Therefore, 
defining $M^*_{N,q,i,\leq}$ and $M^*_{N,q,i,>}$ as in \eqref{eq:def_MNq*} 
with $h$ replaced respectively by $\pr{h_{\leq}}^{\deg}$ and $\pr{h_{>}}^{\deg}$ the equality $M_{N,q,i}\leq  M^*_{N,q,i,\leq}+M^*_{N,q,i,>}$ holds hence by \eqref{eq:bound_deg_Ustat_demo} and Markov's inequality, it follows that
\begin{multline} 
\PP\pr{\max_{2\leq n\leq N}\norm{\sum_{1\leq i<j\leq n}h\pr{X_i,X_j}}_{\Hi}>16x}\\ \leq 
\sum_{i=1}^4\PP\pr{M_{N,q,i,\leq }^*>x}+\sum_{i=1}^4\PP\pr{M_{N,q,i,> }^*>x}
    + \frac{10}{x}qN\sup_{j\geq 2}\E{\norm{h\pr{X_1,X_j}}_{\Hi}}+4N\beta\pr{q}\\
  \leq x^{-r}\sum_{i=1}^4\E{\pr{M_{N,q,i,\leq }^*}^r}
  +x^{-1}\sum_{i=1}^4\E{ M_{N,q,i> }^* }+\frac{10}{x}qN\sup_{j\geq 2}\E{\norm{h\pr{X_1,X_j}}_{\Hi}}+4N\beta\pr{q} .
\end{multline}
We now control the moment of order $r$ of $M_{N,q,1,\leq }^*$ in the following way:
\begin{align}
\E{\pr{M_{N,q,1,\leq }^*}^r}&= \norm{\sum_{\substack{\ell,\ell'=1\\ 0\leq \ell-\ell'\leq q-1}}^{2q}
\max_{1\leq m\leq \ent{\frac{N}{2q}}}
\norm{\sum_{0\leq u<v\leq m}h_{\ell,\ell'}\pr{V^*_{\ell',u},V^*_{\ell',v}   }}_{\Hi}}_r^r
\label{eq:borne_moment_M*_ligne1} \\
& \leq \pr{ \sum_{\substack{\ell,\ell'=1\\ 0\leq \ell-\ell'\leq q-1}}^{2q}\norm{
\max_{1\leq m\leq \ent{\frac{N}{2q}}}
\norm{\sum_{0\leq u<v\leq m}h_{\ell,\ell'}\pr{V^*_{\ell',u},V^*_{\ell',v}   }}_{\Hi}}_r} ^r\label{eq:borne_moment_M*_ligne2}\\
&\leq 4   \pr{ \sum_{\substack{\ell,\ell'=1\\ 0\leq \ell-\ell'\leq q-1}}^{2q}\norm{
\norm{\sum_{0\leq u<v\leq  \ent{\frac{N}{2q}}}h_{\ell,\ell'}\pr{V^*_{\ell',u},V^*_{\ell',v}   }}_{\Hi}}_r}^r\label{eq:borne_moment_M*_ligne3}\\
&\leq C_{1,r}   \pr{ \pr{2q}^2\frac{N}{2q} }^r\E{\norm{\pr{h_{\leq}}^{\deg}\pr{X_1,X'_1}}_{\Hi}^r}\label{eq:borne_moment_M*_ligne4}\\
&\leq C_{2,r}  q^r N^r\E{\norm{h\pr{X_1,X'_1}}_{\Hi}^r\ind{\norm{h\pr{X_1,X'_1}}_{\Hi}\leq R}}\label{eq:borne_moment_M*_ligne5},
\end{align}
where $C_{1,r}$ and $C_{2,r}$ depend only on $r$, the step from \eqref{eq:borne_moment_M*_ligne1}  to \eqref{eq:borne_moment_M*_ligne2} is justified by the triangle inequality, the one from \eqref{eq:borne_moment_M*_ligne2} to \eqref{eq:borne_moment_M*_ligne3} from Doob's inequality, \eqref{eq:borne_moment_M*_ligne4} follows by the use of degeneracy of 
$\pr{h_{\leq}}^{\deg}$, which gives a martingale and reversed martingale property 
for the summation over $j$ and $i$ respectively and a double use of Theorem~4.1 in 
\cite{MR1331198}. Finally, \eqref{eq:borne_moment_M*_ligne5} is a consequence 
of the fact that 
\begin{equation}
\pr{h_\leq}^{\deg}\pr{X_1,X'_1}=h_{\leq}\pr{X_1,X'_1}-\E{h_{\leq}\pr{X_1,X'_1}\mid X_1}-\E{h_{\leq}\pr{X_1,X'_1}\mid X'_1}+\E{h_{\leq}\pr{X_1,X'_1}}
\end{equation}
and an appplication of the triangle inequality. 
A similar bound holds for $\PP\pr{M^*_{N,q,i}>x}$, $i\in\ens{2,3,4}$. 

  The control of the tail of the $U$-statistic associated to $\pr{h_{>}}^{\deg}$ is much simpler and follows from 
 \begin{equation}
 \max_{2\leq n\leq N}\norm{\sum_{1\leq i<j\leq n}\pr{h_{>}}^{\deg} \pr{X_i,X_j}}_{\Hi}
 \leq\sum_{1\leq i<j\leq N} \norm{\pr{h_{>}}^{\deg}  \pr{X_i,X_j}}_{\Hi},
 \end{equation}
 Markov's inequality and 
 \begin{equation*}
\pr{h_>}^{\deg}\pr{X_1,X'_1}=h_{>}\pr{X_1,X'_1}-\E{h_{>}\pr{X_1,X'_1}\mid X_1}-\E{h_{>}\pr{X_1,X'_1}\mid X'_1}+\E{h_{>}\pr{X_1,X'_1}}.
\end{equation*}
 This ens the proof of Proposition~\ref{prop:ineg_dev_Ustat_degeneree}.
\end{proof}
 \subsection{Proof of Theorem~\ref{thm:TLFC}}

We start from decomposition~\ref{eq:Hoeffding}. Theorem~2 in \cite{MR2019054} applied with $X_i$ replaced by $h_1\pr{X_i}$ gives the convergence in distribution in $C_{\Hi}[0,1]$ 
\begin{equation}
\frac 1{\sqrt{n}}\pr{\sum_{j=1}^{\ent{nt}}h_1\pr{X_i}+
\pr{nt-\ent{nt}}h_1\pr{X_{\ent{nt}+1}}}\to W_\Gamma,
\end{equation}
where 
\begin{equation}\label{eq:def_operateur_cov}
\left\langle \Gamma u,v \right\rangle_{\Hi}=\lim_{n\to\infty}\E{  \left\langle n^{-1/2} S_n ,u \right\rangle_{\Hi}\left\langle n^{-1/2} S_n ,v \right\rangle_{\Hi} \mid\sigma\pr{X_k,k\leq 0}}, \quad u,v\in\Hi,
 \end{equation} and $S_n=\sum_{i=1}^nh_1\pr{X_i}$. As Remark~1 in \cite{MR2019054} says, when the sequence $\pr{X_i}_{i\geq 1}$ is ergodic, the 
 operator $\Gamma$ is non random. Ergodicity is ensured by the absolute regularity of $\pr{X_i}_{i\geq 1}$. Then  a computation combined with the absolute convergence of the series in \eqref{eq:def_operateur_Gamma_TLCF} shows that $\Gamma$ has the expression given in that equation.

By \eqref{eq:Hoeffding}, the convergence  \eqref{eq:TLCF} will follows from 
\begin{equation}\label{eq:conv_prob_h2_TLCF}
\frac 1{N^{3/2}}\max_{2\leq n\leq N}\norm{\sum_{1\leq i<j\leq n}h_2\pr{X_i,X_j}  }_{\Hi}
\to 0\mbox{ in probability},
\end{equation}
where $h_2$ is defined as in \eqref{eq:def_h2}, or in other words, that for each positive $\eps$, 
\begin{equation}
\lim_{N\to\infty}\PP\pr{\frac 1{N^{3/2}}\max_{2\leq n\leq N}\norm{\sum_{1\leq i<j\leq n}h_2\pr{X_i,X_j}  }_{\Hi}>\eps}=0.
\end{equation}  
To do so, let us fix $\eta>0$. We use \eqref{eq:ineg_dev_Ustat_degeneree_moment_ordre_2} with $r=2$, $q=\ent{\eta\sqrt{N}}$ and $x=N^{3/2}$ and get that 
\begin{multline}
\PP\pr{\frac 1{N^{3/2}}\max_{2\leq n\leq N}\norm{\sum_{1\leq i<j\leq n}h_2\pr{X_i,X_j}  }_{\Hi}> \eps}
\leq C_2\eps^{-2}N^{-3}\ent{\eta\sqrt{N}}^{2u}N^2\E{H^2 }\\
 +C_2\eps^{-1}N^{-3/2}\ent{\eta\sqrt{N}}N\sup_{j\geq 2}\E{\norm{h\pr{X_1,X_j}}_{\Hi}}+4N\beta\pr{\ent{\eta\sqrt{N}}}.
\end{multline}
Since the assumption $n^2\beta\pr{n}\to 0$ implies that for each fixed $\eta>0$, 
$4N\beta\pr{\ent{\eta\sqrt{N}}}\to 0$, we find that 
\begin{equation*}
\limsup_{N\to\infty}\PP\pr{\frac 1{N^{3/2}}\max_{2\leq n\leq N}\norm{\sum_{1\leq i<j\leq n}h_2\pr{X_i,X_j}  }_{\Hi}>9\eps}
\leq C_2\eps^{-2}\E{H^2}\eta^2+C_2\eps^{-1}\eta\sup_{j\geq 2}\E{\norm{h\pr{X_1,X_j}}_{\Hi}}.
\end{equation*}
Since $\eta$ is arbitrary, the proof of Theorem~\ref{thm:TLFC} is complete.

\subsection{Proof of Theorem~\ref{thm:MSLLN_non_deg_finite_var}}

In view of the decomposition \eqref{eq:Hoeffding}, it suffices to prove that 
\begin{equation}
\lim_{n\to\infty}\frac 1{n^{1+1/p}}\sum_{i=1}^{n-1}\pr{i-1}\pr{h_{1,0}\pr{X_i}-\E{h_{1,0}\pr{X_i}}}=0\mbox{ a.s.},
\end{equation}
\begin{equation}
\lim_{n\to\infty}\frac 1{n^{1+1/p}}\sum_{j=2}^{n}\pr{n-j}\pr{h_{0,1}\pr{X_j}-\E{h_{0,1}\pr{X_j}}}=0\mbox{ a.s.} \mbox{ and }
\end{equation}
\begin{equation}\label{eq:conv_vers_0_LGN_partie_deg}
\lim_{n\to \infty}\frac 1{n^{1+1/p}}\norm{U_n\pr{h_2}}_{\Hi}=0,
\end{equation}
where the functions $h_{1,0}$, $h_{0,1}$ and $h_2$ are 
defined respectively by \eqref{eq:def_h10}, \eqref{eq:def_h01} 
and \eqref{eq:def_h2}. 
Define $S_{i}=\sum_{k=1}^i\pr{h_{1,0}\pr{X_k}-\E{h_{1,0}\pr{X_k}}}$ and  $S'_{j}=\sum_{k=1}^j\pr{h_{0,1}\pr{X_k}-\E{h_{0,1}\pr{X_k}}}$. Then 
\begin{equation}
\sum_{i=1}^{n-1}\pr{i-1}\pr{h_{1,0}\pr{X_i}-\E{h_{1,0}\pr{X_i}}}
=\sum_{j=2}^n \pr{S_n-S_j}
\end{equation}
hence 
\begin{equation}
\frac 1{n^{1+1/p}}\norm{\sum_{i=1}^{n-1}\pr{i-1}\pr{h_{1,0}\pr{X_i}-\E{h_{1,0}\pr{X_i}}}}_{\Hi}\leq \frac{2}{n^{1+1/p}}
\max_{1\leq i\leq n}\norm{S_i}_{\Hi}
\end{equation}
and similarly, 
\begin{equation}
\frac 1{n^{1+1/p}}\norm{\sum_{j=2}^{n}\pr{n-j}\pr{h_{0,1}\pr{X_i}-\E{h_{0,1}\pr{X_i}}}}_{\Hi}\leq \frac{2}{n^{1+1/p}}
\max_{1\leq j\leq n}\norm{S'_j}_{\Hi}.
\end{equation}
By Corollary~3 in \cite{MR2256474}, condition~\ref{assump:mixing_linear_part_LGN} implies that $n^{-1/p}\norm{S_n}_{\Hi}+n^{-1/p}\norm{S'_n}_{\Hi}\to 0$ a.s. Therefore, 
it suffices to prove \eqref{eq:conv_vers_0_LGN_partie_deg}, which reduces, by  the Borel-Cantelli lemma, to prove that for each positive $\eps$, 
\begin{equation}\label{eq:conv_part_deg_Borel_Cantelli}
\sum_{M=0}^\infty \PP\pr{2^{-M\pr{1+\frac 1p}}\max_{2\leq n\leq 2^M}\norm{\sum_{1\leq i<j\leq n}h_2\pr{X_i,X_j}}_{\Hi}> \eps   }<\infty.
\end{equation}
 By \eqref{eq:ineg_dev_Ustat_degeneree_moment_ordre_2} applied with $R=p+\delta$, 
 $N$ replaced by $2^M$, $x$ by $\eps 2^{M\pr{1+1/p}}$ and 
 $q=\ent{2^{Ma}}$ where $a=1/\pr{p+\eta}$, we infer that 
 \begin{multline}\label{eq:ineg_dev_Ustat_degeneree_moment_ordre_2_LGN_1+1/p}
\PP\pr{2^{-M\pr{1+\frac 1p}}\max_{2\leq n\leq 2^M}\norm{\sum_{1\leq i<j\leq n}h_2\pr{X_i,X_j}}_{\Hi}>9\eps   }\leq C_{p+\delta}\eps^{-\pr{p+\delta}}
2^{-2M\pr{1+\frac 1p}}
2^{\frac{2}{p+\eta}M }2^{2M}\E{H^{p+\delta} }\\
 +C_r\eps^{-1}2^{-M\pr{1+\frac 1p}}2^{\frac{M}{p+\eta} }2^M\sup_{j\geq 2}\E{\norm{h\pr{X_1,X_j}}_{\Hi}}+4\cdot 2^M\beta\pr{\ent{2^{\frac{M}{p+\eta}}}}.
 \end{multline}
Since $a<1/p$, the convergence of 
$\sum_{M=0}^{\infty}\PP\pr{2^{-M\pr{1+\frac 1p}}\max_{2\leq n\leq 2^M}\norm{\sum_{1\leq i<j\leq n}h_2\pr{X_i,X_j}}_{\Hi}>9\eps   }$ 
is guaranteed by the convergence of 
$\sum_{M=0}^\infty  2^M\beta\pr{\ent{2^{\frac{M}{p+\eta}}}}$. By dividing the sum over sets of the form $\ens{\ent{2^{Ma}}+1,\dots,\ent{2^{\pr{M+1 }a}}}$ and the fact that $\pr{\beta\pr{k}}_{k\geq 1}$ is non-increasing, the convergence of $\sum_{M=0}^\infty  2^M\beta\pr{\ent{2^{Ma}}}$ is equivalent to that of 
$\sum_{k=1}^\infty  k^{\frac 1a -1}\beta\pr{k}$, which is guaranteed by 
\ref{assump:mixing_linear_part_LGN}.

This ends the proof of Theorem~\ref{thm:MSLLN_non_deg_finite_var}.

\subsection{Proof of Theorem~\ref{thm:MSLLN_non_deg_infinite_var}}
By the same arguments as in the proof of Theorem~\ref{thm:MSLLN_non_deg_finite_var}, 
it suffices to prove that \eqref{eq:conv_part_deg_Borel_Cantelli} holds for each positive $\eps$.  To do so, we apply Proposition~\ref{prop:general_deviation_inequality} with 
$x=2^{M\pr{1+1/p}}\eps$, $N$ replaced by $2^M$, $q=\ent{2^{Ma}}$ and $R=2^{Mb}$, where $a\in (0,1/p)$ and $b>0$ will be specified later and in such a way that 
\begin{equation}\label{eq:demo_thm:MSLLN_non_deg_infinite_var_serie_1}
\sum_{M=0}^\infty 2^{-2M/p}2^{2Ma}\E{H^2\ind{H\leq 2^{Mb}}}<\infty,
\end{equation}
\begin{equation}\label{eq:demo_thm:MSLLN_non_deg_infinite_var_serie_2}
\sum_{M=0}^\infty 2^{-M\pr{1+1/p}}2^{2M} \E{H\ind{H> 2^{Mb}}}<\infty,
\end{equation} 
\begin{equation}\label{eq:demo_thm:MSLLN_non_deg_infinite_var_serie_3}
\sum_{M=0}^\infty 2^M\beta\pr{\ent{2^{Ma}}}<\infty.
\end{equation}
 
Notice that for each positive $b$ and $c$ and each non-negative random variable $Y$, 
\begin{equation}\label{eq:inegalite_pour_sum_2-Nc}
\sum_{M=0}^\infty 2^{-Mc}\E{Y^2\ind{Y\leq 2^{Mb}}}\leq K\pr{b,c}
\E{Y^{2-\frac{c}{b}}},
\end{equation}
\begin{equation}\label{eq:inegalite_pour_sum_2Nc}
\sum_{M=0}^\infty 2^{Mc}\E{Y\ind{Y>2^{Mb}}}\leq K\pr{b,c}
\E{Y^{1+\frac{c}{b}}}.
\end{equation}

Therefore, applying \eqref{eq:inegalite_pour_sum_2Nc} with $c=1-1/p$ 
imposes the choice $b=\pr{p-1}/\pr{p\pr{p+\delta-1}}$. 
Now, applying \eqref{eq:inegalite_pour_sum_2-Nc} with $c=2\pr{1+1/p}
-2a$ shows that $a$ must satisfy 
\begin{equation}
2-\frac{2/p-2a}b=p+\delta,
\end{equation}
hence 
\begin{equation}
a=\frac{p\pr{p-1}+\delta\pr{p+1}}{2p\pr{p+\delta-1}}.
\end{equation}
Notice that \eqref{eq:demo_thm:MSLLN_non_deg_infinite_var_serie_3} is 
equivalent to the convergence of $\sum_{k=1}^{\infty}k^{1/a-1}\beta\pr{k}$. Therefore,  \eqref{eq:conv_part_deg_Borel_Cantelli} is a consequence of 
\eqref{eq:cond_melange_non_deg_infinite_var}.

\subsection{Proof of Theorem~\ref{thm:MSLLN_deg_finite_var}}

By the Borel-Cantelli lemma, it suffices to prove that 
\begin{equation}\label{eq:deg_BC}
\sum_{M=0}^\infty \PP\pr{2^{-2M/p}\max_{2\leq n\leq 2^M}\norm{\sum_{1\leq i<j\leq n}h\pr{X_i,X_j}}_{\Hi}> \eps   }<\infty.
\end{equation}

To do so, let $M$ and $\eps>0$ and let us apply Proposition~\ref{prop:general_deviation_inequality} with $r=2$,
$x=2^{2M/p}\eps$, $N$ replaced by $2^M$ and $q=\ent{2^{Ma}}$, where 
$a=1/\pr{\eta+p/\pr{2-p}}$. The convergence of the series 
$\sum_{M=0}^\infty 2^M\beta\pr{\ent{2^{Ma}}}$ is equivalent to that of 
$\sum_{k=1}^\infty k^{-1+1/a}\beta\pr{k}$, which is guaranteed by 
 \eqref{eq:assum_mix_MSLLN_deg_infinite_var}.

\subsection{Proof of Theorem~\ref{thm:MSLLN_deg_infinite_var}}
Here again, it suffices to check \eqref{eq:deg_BC}. Letting $a=\delta/\pr{p\pr{p+\delta-1}}$ and $b:=p\pr{p+\delta-1}/\pr{2\pr{p-1}}$, then applying Proposition~\ref{prop:ineg_dev_Ustat_degeneree} with $N$ replaced by $2^M$, $r=2$,
$x=2^{2M/p}\eps$, $q=\ent{2^{Ma}}$ and $R=2^{Mb}$ gives the convergence of 
$\sum_{M=0}^\infty \PP\pr{2^{-2M/p}\max_{2\leq n\leq 2^M}\norm{\sum_{1\leq i<j\leq n}h\pr{X_i,X_j}}_{\Hi}> \eps   }$ thanks to \eqref{eq:assum_mix_MSLLN_deg_infinite_var}.

\end{document}